\documentclass[12pt]{amsart}
\usepackage{amsmath,amssymb,amsbsy,amsfonts,latexsym,amsopn,amstext,cite,
                                               amsxtra,euscript,amscd,bm,mathabx}
\usepackage{url}
\usepackage[colorlinks,linkcolor=blue,anchorcolor=blue,citecolor=blue,backref=page]{hyperref}
\usepackage{color}
\usepackage{graphics,epsfig}
\usepackage{graphicx}
\usepackage{float}
\usepackage{bbm}
\usepackage{appendix}
\usepackage{epigraph}

\hypersetup{breaklinks=true}

\usepackage[norefs,nocites]{refcheck}
\usepackage[english]{babel}
\begin{document}

\newtheorem{thm}{Theorem}
\newtheorem{lem}[thm]{Lemma}
\newtheorem{claim}[thm]{Claim}
\newtheorem{cor}[thm]{Corollary}
\newtheorem{prop}[thm]{Proposition} 
\newtheorem{definition}[thm]{Definition}
\newtheorem{question}[thm]{Open Question}
\newtheorem{conj}[thm]{Conjecture}
\newtheorem{rem}[thm]{Remark}
\newtheorem{prob}{Problem}
\newtheorem{Claim}{Claim}
\newtheorem*{thank}{\ \ \ \bf Acknowledgment}

\newtheorem{lemma}[thm]{Lemma}

\newcommand{\hh}{{{\mathrm h}}}

\numberwithin{equation}{section}
\numberwithin{thm}{section}
\numberwithin{table}{section}

\def\vol {{\mathrm{vol\,}}}
\def\squareforqed{\hbox{\rlap{$\sqcap$}$\sqcup$}}
\def\qed{\ifmmode\squareforqed\else{\unskip\nobreak\hfil
\penalty50\hskip1em\null\nobreak\hfil\squareforqed
\parfillskip=0pt\finalhyphendemerits=0\endgraf}\fi}

\def \balpha{\bm{\alpha}}
\def \bbeta{\bm{\beta}}
\def \bgamma{\bm{\gamma}}
\def \blambda{\bm{\lambda}}
\def \bchi{\bm{\chi}}
\def \bphi{\bm{\varphi}}
\def \bpsi{\bm{\psi}}
\def \bomega{\bm{\omega}}
\def \btheta{\bm{\vartheta}}

\newcommand{\bfxi}{{\boldsymbol{\xi}}}
\newcommand{\bfrho}{{\boldsymbol{\rho}}}

\def\Kab{\cK_\psi(a,b)}
\def\Kuv{\cK_\psi(u,v)}
\def\SAUV{\cS_\psi(\balpha;\cU,\cV)}
\def\SAAV{\cS_\psi(\balpha;\cA,\cV)}

\def\SUV{\cS_\psi(\cU,\cV)}
\def\SAB{\cS_\psi(\cA,\cB)}

\def\Kmnp{\cK_p(m,n)}

\def\KKap{\cH_p(a)}
\def\KKaq{\cH_q(a)}
\def\KKmnp{\cH_p(m,n)}
\def\KKmnq{\cH_q(m,n)}

\def\Klmnp{\cK_p(\ell, m,n)}
\def\Klmnq{\cK_q(\ell, m,n)}

\def \SALMNq {\cS_q(\balpha;\cL,\cI,\cJ)}
\def \SALMNp {\cS_p(\balpha;\cL,\cI,\cJ)}

\def \SACXMQX {\fS(\balpha,\bzeta, \bxi; M,Q,X)}

\def\SAMJp{\cS_p(\balpha;\cM,\cJ)}
\def\SAMJq{\cS_q(\balpha;\cM,\cJ)}
\def\SAqMJq{\cS_q(\balpha_q;\cM,\cJ)}
\def\SAJq{\cS_q(\balpha;\cJ)}
\def\SAqJq{\cS_q(\balpha_q;\cJ)}
\def\SAIJp{\cS_p(\balpha;\cI,\cJ)}
\def\SAIJq{\cS_q(\balpha;\cI,\cJ)}

\def\RIJp{\cR_p(\cI,\cJ)}
\def\RIJq{\cR_q(\cI,\cJ)}

\def\TWXJp{\cT_p(\bomega;\cX,\cJ)}
\def\TWXJq{\cT_q(\bomega;\cX,\cJ)}
\def\TWpXJp{\cT_p(\bomega_p;\cX,\cJ)}
\def\TWqXJq{\cT_q(\bomega_q;\cX,\cJ)}
\def\TWJq{\cT_q(\bomega;\cJ)}
\def\TWqJq{\cT_q(\bomega_q;\cJ)}

 \def \xbar{\overline x}
  \def \ybar{\overline y}

\def\cA{{\mathcal A}}
\def\cB{{\mathcal B}}
\def\cC{{\mathcal C}}
\def\cD{{\mathcal D}}
\def\cE{{\mathcal E}}
\def\cF{{\mathcal F}}
\def\cG{{\mathcal G}}
\def\cH{{\mathcal H}}
\def\cI{{\mathcal I}}
\def\cJ{{\mathcal J}}
\def\cK{{\mathcal K}}
\def\cL{{\mathcal L}}
\def\cM{{\mathcal M}}
\def\cN{{\mathcal N}}
\def\cO{{\mathcal O}}
\def\cP{{\mathcal P}}
\def\cQ{{\mathcal Q}}
\def\cR{{\mathcal R}}
\def\cS{{\mathcal S}}
\def\cT{{\mathcal T}}
\def\cU{{\mathcal U}}
\def\cV{{\mathcal V}}
\def\cW{{\mathcal W}}
\def\cX{{\mathcal X}}
\def\cY{{\mathcal Y}}
\def\cZ{{\mathcal Z}}

\def\NmQR{N(m;Q,R)}
\def\VmQR{\cV(m;Q,R)}

\def\Xm{\cX_m}

\def \A {{\mathbb A}}
\def \B {{\mathbb B}}
\def \C {{\mathbb C}}
\def \F {{\mathbb F}}
\def \G {{\mathbb G}}
\def \L {{\mathbb L}}
\def \K {{\mathbb K}}
\def \N {{\mathbb N}}
\def \Q {{\mathbb Q}}
\def \R {{\mathbb R}}
\def \Z {{\mathbb Z}}
\def \T {{\mathbb T}}
\def \fS{\mathfrak S}

\def\e{{\mathbf{\,e}}}
\def\ep{{\mathbf{\,e}}_p}
\def\eq{{\mathbf{\,e}}_q}

\def\\{\cr}
\def\({\left(}
\def\){\right)}
\def\fl#1{\left\lfloor#1\right\rfloor}
\def\rf#1{\left\lceil#1\right\rceil}

\def\Tr{{\mathrm{Tr}}}
\def\Im{{\mathrm{Im}}}

\def \bFp {\overline \F_p}

\newcommand{\pfrac}[2]{{\left(\frac{#1}{#2}\right)}}

\def \Prob{{\mathrm {}}}
\def\e{\mathbf{e}}
\def\ep{{\mathbf{\,e}}_p}
\def\epp{{\mathbf{\,e}}_{p^2}}
\def\em{{\mathbf{\,e}}_m}

\def\Res{\mathrm{Res}}
\def\Orb{\mathrm{Orb}}

\def\vec#1{\mathbf{#1}}
\def\flp#1{{\left\langle#1\right\rangle}_p}

\def\mand{\qquad\mbox{and}\qquad}

\def\vec#1{\mathbf{#1}}
\def\flp#1{{\left\langle#1\right\rangle}_p}
\def \ds {\displaystyle}
\def\mand{\qquad\mbox{and}\qquad}
\def\build#1_#2^#3{\mathrel{\mathop{\kern 0pt#1}\limits_{#2}^{#3}}}
\def\tend#1#2{\build\hbox to 12mm{\rightarrowfill}_{#1\rightarrow #2}^{ }}

\newcommand{\commI}[1]{\marginpar{%
    \vskip-\baselineskip 
    \raggedright\footnotesize
    \itshape\hrule\smallskip\begin{color}{red}#1\end{color}\par\smallskip\hrule}}

\newcommand{\commO}[1]{\marginpar{%
    \vskip-\baselineskip 
    \raggedright\footnotesize
    \itshape\hrule\smallskip\begin{color}{blue}#1\end{color}\par\smallskip\hrule}}


\title[A disproof of $L^\alpha$ Rudin conjecture]
{A disproof of $L^\alpha$ polynomials Rudin conjecture, $2 \leq \alpha<4.$ 
}

\author[\MakeLowercase{e.} H. \MakeLowercase{el} Abdalaoui]{\MakeLowercase{e.} H. \MakeLowercase{el} Abdalaoui} 
\address{Normandy University of Rouen,
	Department of Mathematics, LMRS  UMR 6085 CNRS, Avenue de l'Universit\'e, BP.12,
	76801 Saint Etienne du Rouvray - France .}
\email{elhoucein.elabdalaoui@univ-rouen.fr}


\date{\today}

\begin{abstract}   It is shown that the $L^\alpha$-norms polynomials Rudin conjecture fails. Our counterexample is inspired by Bourgain's work on NLS. Precisely, his study of the Strichartz's inequality of the $L^6$-norm of the periodic solutions given by the two dimension Weyl sums. We gives also a lower bound of the $L^\alpha$-norm of such solutions for 
$\alpha \neq 2$. As a consequence, we establish that for any $0<a<b,$ the following set 
$E(a,b)=\Big\{(x,t) \in \T^2 \; : \;a \sqrt{N} \leq \Big|\sum_{n=1}^{N}e(n^2t+nx)\Big| 
	\leq b \sqrt{N} \;
	\textrm{~infinitely~often}\;\Big\},$ 
has a Lebesgue measure $0$. We further present an alternative proof of Cordoba's theorem based on Paley-Littlewood inequalities.
\end{abstract}

\subjclass[2010]{37A20, 35Q55, 37A44, 11B25, 11B57}

\keywords{Rudin conjecture on square, Weyl sums, Vinogradov method, circle method, NLSE, Strichartz inequality, periodic solutions.}

\maketitle

\epigraph{Fej\'{e}r used to say-in the 1930's, "Everybody writes and nobody reads." This was true eventhen. Reviewing has improved, but even so it is very hard.}{\textit{ P\'{a}l  Erd\"{o}s}}

\epigraph{The purpose of life is to conjecture and prove..} {\textit{ P\'{a}l   Erd\"{o}s}}

\epigraph{The man who is seeking truth is free of all societies and cultures.}{\textit{ Jiddu Krishnamurti}}

\section{Introduction}
The purpose of this note is to present a  counterexamples to the so called Rudin conjecture on the quadratic trigonometric sums based on an estimation of the following Weyl sums
$$\sum_{n=1}^{N}e^{inx}e^{in^2t}.$$

Those sums are very well study and has many connection to various areas of mathematics including NonLinear Shr\"{o}dinger Equations (NLSE). For an application in Number Theory, we refer to \cite[p.196]{IK}. Therein, the upper bound of the $\alpha$-norms of such sums is given for $\alpha \in \{4,6\}.$ Here, we will gives an estimation of lower bound for all $\alpha>2.$

According to Cordoba \cite{Cordoba}, Rudin conjecture is a special case of the following old conjecture in the theory of Fourier series. 

\begin{conj}Let $k$ be a positive integer and $(a_n)_{n \in \Z}$ a sequence of complex numbers. Define the Fourier series $S(t)$ by 
	$\ds S(t) \sim \sum_{ n } a_n e^{2\pi i n^k \theta}.$ 
	Then the $L^2$ norm and the $L^\alpha$ norm of $S(t)$, for $\alpha <2k$, are equivalent, that is, there exist $C_\alpha$ such that
	$$\Big\|\sum_{ n \in \Z} a_n e^{2\pi i n^k \theta}\Big\|_\alpha \leq C_\alpha \Big(\sum_{ n \in \Z} \big|a_n\big|^2 \Big)^{\frac12}.$$ 		
\end{conj}

Precisely, Rudin in his seminal paper \cite{Rudin} (see the end of section 4.6) asked the following.

For any $\alpha \in [0,4)$, is it possible to find a constant $C_\alpha$, such that 
for any $N \in \N^*$, we have
$$\Big\|\sum_{k=1}^{N}a_k e(k^2\theta)\Big\|_\alpha \leq C_\alpha
\Big\|\sum_{k=1}^{N}a_k e(k^2\theta)\Big\|_2 
$$ 
where, as customary, $\T$ denote the circle and $e(t)=e^{2\pi i t},$ $t \in \R$.

By a similar analogy due to P. Cohen on $e^{i n_j \theta}$, where $\frac{n_{j+1}}{n_j}>\lambda>1$ \cite[p.192]{Cohen}, this conjecture can be related to the famous Kintchine's and Marcinkiewicz-Zygmund inequalities on the equivalence of the $L^p$-norm of the sums of independent random variables.

It is well-known that that Rudin conjecture holds for the trivial case $a_k=1$ and  for
the monotonically decreasing sequence $(a_k)$ \cite{Cordoba}. This can be extended to the case of the monotonically increasing sequence $(a_k)$, and to the case $(k^{2\ell} b_k)$, for any monotonic sequence $(b_k)$, by applying Bernstein-Zygmund inequalities. This was observed by e. el Abdalaoui and I. Shparlinski \cite{AS}. For sake of completence, we present a self-contain proof of Cordoba's theorem and its extension in the appendix. 

Let us remind that that Rudin conjecture implies that $Q_2(a,q;N)=O(N^{\frac12+\epsilon})$, uniformly on $a$ and $q$ (\cite[Theorem 3.5]{Rudin}), where $Q_2(a,q;N)$ denote the number of perfect square
in the arithmetic progression $a+qn$, $n=1, \ldots, N$. In his famous problems paper \cite{Erdos},  Erd\"{o}s conjectured \cite[Problem 16]{Erdos}that $Q_2(a,q;N)=o(N)$. This was solved by Szm\'eredi \cite{Sze}  as a consequence  of  Fermat's  result which say that no four squares in arithmetic progression over $\Z$ exist. Bombieri, Granville and Pintz in \cite{BGP} improved this result and established that $Q_2(a,q;N)=(N^{2/3}{(\log(N))}^A$ for a suitable constant $A$. Subsequently, Bombieri and Zannier \cite{BZ} proved that $Q_2(a,q;N)=O(N^{3/5}{(\log(N))}^A)$ for a suitable constant $A$.
However, here we will prove the following.

\begin{thm}\label{main1}The Rudin conjecture is not true.
\end{thm}

Of-course, Theorem \ref{main1} bring no information about $Q_2(a,q;N)$-Rudin conjecture, that is, $Q_2(a,q;N)=O(\sqrt{N})$, uniformly on $a$ and $q$ neither about its weaker forms, that is,
$Q_2(a,q;N)=O(N^{\frac{1}{2}+\varepsilon})$, uniformly on $a$ and $q$.

The proof of Theorem \ref{main1} is based on Bourgain strategy to estimate the $L^6$-norm of the Weyl sums in his 1993 NLS's paper \cite{B93}. Therein, Bourgain extend the Stichartz's inequality (see Remark 2, page 118.). Here, inspired by his work, we will prove the following Theorem.
\begin{thm}\label{main2}Let $\epsilon$ be a very small positive (less than $\frac1{100}$) and  $\alpha=4-\epsilon$, there is a constant $C_\alpha>0$ such that
\begin{align*}
\sup_{x \in \T}\Big(\int_{0}^{1} \Big|\sum_{n=1}^{N}e(nx) e(n^2t)\Big|^\alpha dt\Big)\geq C_\alpha N^{2}.
\end{align*}	
\end{thm} 
We will further establish the following lower bound for the $L^\alpha$-norm of such solutions for  $\alpha \neq 2$.
\begin{thm}\label{main3}Let $\alpha >2$. Then,  there is a constant $C_\alpha>0$ such that
	\begin{align*}
	\int_{0}^{1}\int_{0}^{1} \Big|\sum_{n=1}^{N}e(nx) e(n^2t)\Big|^\alpha dt \geq 
	\begin{cases}
	C_\alpha N^{\frac{3}{4}\alpha-\frac{3}{2}} & \textrm{if}\;  \alpha \neq 6.\\
	C_\alpha N^{3}\log(N) & \textrm{if~not~} \; \alpha =6.
	\end{cases}
	\end{align*}	
\end{thm} 
\noindent{}Consequently, we have the following corollary.
\begin{cor}Let $C,K>0$ be a two positive constant, and put 
\begin{eqnarray*}
	E=\Big\{(x,t) \in \T^2 \; : \;C \sqrt{N} \leq \Big|\sum_{n=1}^{N}e(n^2t+nx)\Big| 
	\leq K \sqrt{N} \;\\
	\textrm{for~infinity~many }\; N \Big\}.
\end{eqnarray*}
\noindent{}Then $E$ is a measurable set with respect to the Lebesgue measure on $\T^2$ and its measure is zero, that is, $dx \otimes dt(E)=|E|=0.$
\end{cor}
\begin{proof}We proceed by contradiction, assume that $|E|>0, $ and let $t\in [0,1).$ Put $T_{t}(x,y,z)=(x+t,y+2x+t,z+e(y))$. $T_t$ is in the class of  Furstenberg-Type maps, and we have  
$T_t^n(\frac{x}{2},0,0)=(\frac{x}{2}+nt,n^2t+nx,\sum_{k=0}^{n}e(k^2t+kx)).$ It is well-known that for a $G_\delta$ set of $t$ we have that $T_{t}$ is an ergodic measure-preserving transformation on $\T^2 \times \C$ \cite{GNV}, but this set has measure zero \footnote{We ask if Greshchonig-Nerurkar-Voln\'{y} theorem can be improved by exhibiting a set of positive Lebesgue measure of $t$ for which the maps $T_t$ still ergodic.}. Thus, we will used an ergodic decomposition. Therefore, the set $E_t=\Big\{x | (x,t) \in E\Big\}$ is  
	$T_{t}$-invariant set with positive measure (take the function $\pi_3(x,y,z)=z$ and observe that 
	$E=\Big\{(x,t):\overline{\lim}\big|\frac{1}{\sqrt{N}}\sum_{k=0}^{N}e(k^2t+kx))\big|>a\Big\}\bigcap
	\Big\{(x,t): \underline{\lim} \big|\frac{1}{\sqrt{N}}\sum_{k=0}^{N}e(k^2t+kx))\big|<b\Big\}$). Hence, its Lebesgue measure is $1$ and thus $|E|=1$. But, obviously, we have 
	$$\int_{0}^{1} \int_{0}^{1} \frac{1}{N^3\log(N)^{\frac{1}{6}}}\Big|\sum_{n=1}^{N}e(n^2t+nx)\Big|^2 dx dt \tend{N}{+\infty}0. $$  
since $\big\{e(nx), n\in \Z\big\}$ is an orthonormal family. One can used also Bourgain's result which assert that ${(n,n^2)}$ is a $\Lambda_4$-set \cite{B93}, that is, there is a constant $c>0$ such that 
$$\int_{0}^{1} \int_{0}^{1} \frac{1}{N^2}\Big|\sum_{n=1}^{N}e(n^2t+nx)\Big|^4  dx dt \leq c. $$
 Therefore, there exists a subsequence $(N_k)$ for which almost all $(x,t)\in \T^2,$ 
	$$\frac{1}{N_k^3(\log(N_k))^{\frac{1}{6}}}\Big|\sum_{n=1}^{N_k}e(n^2t+nx)\Big|^6 \tend{k}{+\infty}0.$$
	and, by applying the Lebesgue Dominated Convergence Theorem on $E$, we get
	$$\int_{\T^2}\frac{1}{N_k^3\log(N_k)}\Big|\sum_{n=1}^{N_k}e(n^2t+nx)\Big|^6 dx dt\tend{k}{+\infty}0,$$
	which contradict \eqref{lsix-3}, and the proof of the corollary is complete. 
\end{proof}
\begin{rem}The set of $t \in [0,1)$ such that for each $x \in \T$, $(x,t) \in E$, contain the set of numbers with bounded partial quotients, by a theorem due to Hardy-Littlewood \cite[Theorem 2.25]{HL}. We would like to mention also that by Corollary 2 from \cite{F}, we have that $E$ is contain in a set of second category of Baire.
\end{rem}
Following Bourgain	ideas, we will used some idea from circle method and the classical Gauss estimation combined with ven der Corput method. Before proceeding to the proof, let us observe that  the proof of Theorem \ref{main1} will follows from Theorem \ref{main2}, we will give it in the section \ref{second}. The proof of Theorem \ref{main2} is the subject of the section \ref{third}. 
\section{Proof of the main result.}\label{second}
In this section, we proceed to the proof of Theorem \ref{main1}. Assume that Rudin conjecture is true. Then, there is a positive constant $K_\alpha$ such that, for any $N \geq 1$, for any complex sequence $(a_n)$, we have
\begin{align*}
	\Big(\int_{0}^{1} \Big|\sum_{n=1}^{N}a_n e(n^2t)\Big|^\alpha dt\Big)^{\frac{1}{\alpha}} \leq K_\alpha \sqrt{N}.
\end{align*}  
Take $a_n=e(nx)$, for $x \in [0,1)$. Then 
\begin{align*}
\int_{0}^{1} \Big|\sum_{n=1}^{N}e(nx) e(n^2t)\Big|^\alpha dt\leq K_\alpha N^{\frac{\alpha}{2}}.
\end{align*}
Whence, by taking the supremum, we get
\begin{align*}
\sup_{x \in \T}\Big(\int_{0}^{1} \Big|\sum_{n=1}^{N}e(nx) e(n^2t)\Big|^\alpha dt\Big) \leq K_\alpha N^{\frac{\alpha}{2}}.
\end{align*}
Now, taking into account Theorem \ref{main2}, we obtain
$$ C_\alpha N^{2} \leq K_\alpha N^{\frac{\alpha}{2}}.$$

$$ N^{2-\frac{\alpha}{2}} \leq K'_{\alpha}.$$
Which is impossible since $\frac{\alpha}{2}<2$. This complete the proof of Theorem \ref{main1}.
\section{Proof of Theorem \ref{main2}}\label{third}
The fundamental idea in the proof of Theorem \ref{main2} is based on the circle method combined with van der Corput type argument and the theory of Gauss sums. In the proof, we will present with more details Bourgain's observation in page 118 of his 1993's paper \cite{B93}. We will thus follows Hardy-Littlewood circle method. For a nice account on it, we refer to \cite{Vau} or \cite{Vin}. Let $\alpha=4-\epsilon$ and $\epsilon <10^{-2}$. Define the major arcs by 
\begin{align}\label{major}
&\mathcal{M}(q,a,b)=\nonumber\\
&\Big\{(x,t) \in[0,1)^2\;:\;\big|x-\frac{b}{q}\big|<10^{-2}.N^{\epsilon-1},\;\;
\big|t-\frac{a}{q}\big|<10^{-2}N^{\epsilon-2}\Big\}, \nonumber \\
&\textrm{with}\; \; 1 \leq a <q \leq N^{\frac{1}{2}-\epsilon},\; a \wedge q=1, 0 \leq b <q.\; \; \; \; \; \; \; \; 
\end{align}
We notice that the major arc satisfy
\begin{align*}
\mathcal{M}(q,a,b)= I(q,b)\times I(a,q)\;\; \textrm{with}\; \; \\
I(b,q)=\Big[\frac{b}{q}-10^{-2}N^{\epsilon-1},\frac{b}{q}+10^{-2}N^{\epsilon-1}\Big] \;\;
\textrm{and}\; \;\\ I(a,q)=\Big[\frac{a}{q}-10^{-2}N^{\epsilon-2},\frac{a}{q}+10^{-2}N^{\epsilon-2}\Big]
\end{align*}
It is well know that the major arcs are disjoints \cite{Vau}, that is, 
$$\mathcal{M}(q,a,b) \cap \mathcal{M}(q',a',b') \neq \emptyset \Longrightarrow q=q', a=a' \;\textrm{and}\;\; b=b'.$$
In this setting, we have the following crucial lemma in the proof.
\begin{lem}\label{Key}Let $N$ be a positive integer, $\alpha=4-\epsilon$ and $\epsilon <10^{-2}$, let  
$1 \leq a <q \leq N^{\frac{1}{2}-\epsilon},\; a \wedge q=1, 0 \leq b <q.$ Assume $q$ is odd or $q \equiv 0$ mod $4$ and $b$ is even, or $q \equiv 2$ mod $4$ and $b$ is odd. Then, for any $(x,t) \in  \mathcal{M}(q,a,b)$, we have
\begin{align*}
\Big|\sum_{n=0}^{N}e(nx) e(n^2t)\Big| \succsim\frac{N}{\sqrt{q}}.
\end{align*}
\end{lem}
For the proof of Lemma \ref{Key}, we need the following classical lemma on the generalized Gauss sums (see for instance \cite[p.93]{GK}). We recall that the the generalized Gauss sums are given by
$$S(a,b,q)=\sum_{n=1}^{q}e{\Big(\frac{a n^2+bn}{q}\Big)},$$
such sum is invariant under the shift.
\begin{lem}\label{Gauss}Let $a,q $ be relatively prime natural numbers and $b \in \Z$. Then, we have
\begin{align*}
\big|S(a,2b,q)\big|=\begin{cases}
\sqrt{q} & \textrm{if $q$ is odd}\\
0, & \textrm{if $q \equiv 2$ mod 4}\\  
\sqrt{2q} & \textrm{if $q \equiv 0$ mod 4},
\end{cases}, \\
\textrm{and}\;\;\;\;\;\;\;\;\;\;\;\;\;\;\;\;\;\;\;\;\;\;\;\;\;\;\;\;\;\;\;\;\; \;\;\;\;\;\;\;\;\;\;\;\;\;\;\;\;\;\;\;\;\;\;\;\;\;\;\;\;\;\;\;\;\;\;\;\;\;\;\;\;\;\;\;\;\;\;\;\;\;\;\;\;\;\;\;\; \; \; \; \; \; \; \; \; \; \; \; \nonumber\\
\big|S(a,2b+1,q)\big|=\begin{cases}
\sqrt{q} & \textrm{if $q$ is odd}\\
\sqrt{2q}, & \textrm{if $q \equiv 2$ mod 4}\\  
0 & \textrm{if $q \equiv 0$ mod 4}.
\end{cases}
\end{align*}
\end{lem}
Let us point out that by Theorem 8.1 from \cite[p.200]{IK}, for each $(x,t) \in \mathcal{M}(q,a,b)$, we have
\begin{align*}
\Big|\sum_{n=0}^{N}e(nx) e(n^2t)\Big| &\leq 2\frac{N}{\sqrt{q}}+\sqrt{q}\log(q).
\end{align*}
At this point we present the proof of Theorem \ref{main2}. 
\begin{proof}[\textbf{Proof of Theorem \ref{main2}.}]
Let $a,b,q$ as in Lemma \ref{Key}, and $\alpha=4-\epsilon$. Then
\begin{align*}
\sup_{x \in \T}\Big(\int_{0}^{1} \Big|\sum_{n=1}^{N}e(nx) e(n^2t)\Big|^\alpha dt\Big)
&\geq \sup_{x \in I(b,q)}\Big(\int_{I(a,q)} \Big|\sum_{n=1}^{N}e(nx) e(n^2t)\Big|^\alpha dt\Big) \nonumber\\
&\geq \frac{N^{\alpha+\epsilon-2}}{q^{\frac{\alpha}{2}}}=
\frac{N^{2}}{q^{\frac{\alpha}{2}}}.	
\end{align*}
Since $q$ was arbitrary in $(2,N^{\frac{1}{2}-\epsilon}]$. We get 
\begin{align*}
\sup_{x \in \T}\Big(\int_{0}^{1} \Big|\sum_{n=1}^{N}e(nx) e(n^2t)\Big|^\alpha dt\Big)
&\geq c_\alpha .N^{2}.
\end{align*}
The proof of the theorem is complete.
\end{proof}
We still need to give the proof of our fundamental lemma \ref{Key}.
\begin{proof}[\textbf{Proof of Lemma \ref{Key}.}]Put $\tau=t-\frac{a}{q}$ and $\xi=x-\frac{b}{q}.$ f. By the division algorithm we can write $N = q [\frac{N}{q}] + r$, $ 0 \leq r < q.$ Therefore
	\begin{align}\label{D1}
	\sum_{ n=1 }^{N}e(n^2t+nx)=\sum_{ n=1 }^{q [\frac{N}{q}] }e(n^2t+nx)+O(q),
	\end{align}
since $r < q$ and $|e(n^2t+nx)|=1.$ Applying again the division algorithm  to write 
$n=mq+s,$ $ 0 \leq s < q.$ We rewrite \ref{D1} as follows
 \begin{align}
 &\sum_{ n=1 }^{N}e(n^2t+nx) \nonumber\\
 &=\sum_{s=1}^{q} \sum_{m=1}^{[\frac{N}{q}]}e\Big((mq+s)^2\Big(\frac{a}{q}+\tau\Big)+(mq+s)\Big(\frac{b}{q}+\xi\Big)\Big)+O(q) \nonumber\\
 &=\sum_{s=1}^{q}e\Big(s^2\frac{a}{q}+s\frac{b}{q}\Big)\sum_{m=1}^{[\frac{N}{q}]}e\Big((mq+s)^2\tau+(mq+s)\xi\Big)+O(q) \label{D2}
 \end{align}
The last equality is due to the fact that mod $q$, we have 
\begin{align*}
mq+s \equiv s \; \; \textrm{and} \;\; (mq+s)^2 \equiv s^2,
\end{align*}
Moreover, under our assumption \ref{major}, the error term satisfy 
$$O(q) \lesssim N^{\frac{1}{2}} \ll N^{\frac{3}{4}+\frac{\epsilon}{2}} <\frac{N}{\sqrt{q}}.$$
Now, by applying van der Corput type argument, we claim that for each $s \in[1,q]$, we have
\begin{align}\label{error}
	\sum_{m=1}^{\frac{N}{q}}e\Big((mq+s)^2\tau+(mq+s)\xi\Big) \sim\frac{N}{q},
\end{align}
The estimation \ref{error} is given up to some errors to be precised later.

Indeed, by appealing to Euler summation formula (see \cite[Theorem 3.1]{Apostol} or \cite[eq. (4.8), p.40]{Vau}, we have
\begin{align}\label{error2}
&\sum_{m=1}^{\frac{N}{q}}e\Big((mq+s)^2\tau+(mq+s)\xi\Big)=\int_{1}^{\frac{N}{q}}e\Big(\big(yq+s\big)^2\tau+\big(yq+s\big)\xi\Big) dy\nonumber \\
&+2\pi i\int_{1}^{\frac{N}{q}}\big(y-[y])\Big(2q(yq+s)\tau+\xi\Big)e\Big(\big(yq+s\big)^2\tau+\big(yq+s\big)\xi\Big) dy
\end{align} 
Moreover, by changing the variable of integration to $z=yq+s$, we can rewrite \eqref{error2} as follows
\begin{align}\label{error3}
&\sum_{m=1}^{\frac{N}{q}}e\Big((mq+s)^2\tau+(mq+s)\xi\Big)=\frac{1}{q}\int_{q+s}^{N+s}e\Big(z^2\tau+z\xi\Big) dz\nonumber \\
&+\frac{2\pi i}{q}\int_{q+s}^{N+s}\Big(\frac{z-s}{q}-\Big[\frac{z-s}{q}\Big]\Big)
\Big(2qz\tau+\xi\Big)e\Big(z^2\tau+z\xi\Big) dz
\end{align}
Now, we estimate the second term in \eqref{error3} as follows
\begin{align*}
&\Big|\frac{2\pi i}{q}\int_{q+s}^{N+s}\Big(\frac{z-s}{q}-\Big[\frac{z-s}{q}\Big]\Big)
\Big(2qz\tau+\xi\Big) e\Big(z^2\tau+z\xi\Big) dz\Big|\\
&\leq \frac{2\pi}{q} \big(2q(N+s)|\tau|+|\xi|\big)\\
&\leq \frac{2\pi}{q} \Big(4 N^2 .10^{-2}N^{\epsilon-2}+N .10^{-2}N^{\epsilon-1}\Big)\\
&\leq \frac{10\pi}{q} . N^{\epsilon} \ll \frac{N}{q}.
\end{align*}
since $|\tau| \leq 10^{-2}N^{\epsilon-2}, |\xi| \leq 10^{-2}N^{\epsilon-1} \frac{1}{N^{\epsilon-2}}$, with $q \leq N^{\frac{1}{2}-\epsilon}$. Applying again the same arguments; we estimate the first term in \eqref{error2} as follows 
\begin{align*}
&\Big|\frac{1}{q}\int_{q+s}^{N+s}
\Big(e\Big(z^2\tau+z\xi\Big)-1\Big) dz\Big| \label{erro5}\\
&\leq \frac{1}{q} \big((N+s)^2|\tau|+(N+s)\xi\big)\\
&\leq \frac{1}{q} \Big(4 N^2 .10^{-2}N^{\epsilon-2}+2 N .10^{-2}N^{\epsilon-1}\Big)\\
&\leq \frac{6}{q} . N^{\epsilon} \ll \frac{N}{q}. 
\end{align*}
Notice that \eqref{erro5} it is due to the fact that $|\sin(x)| \leq |x|$ for all $x \in \R$.

\noindent{}Summarizing, under our assumption, we have proved that informally for each $s \in \{1,\cdots,q\},$ we have
\begin{align}
	\Big|\sum_{m=1}^{\frac{N}{q}}e\Big((mq+s)^2\tau+(mq+s)\xi\Big)\Big| =
	\frac{N}{q}+O\Big(\frac{N^{\epsilon}}{q}\Big)
\end{align}
This combined with \eqref{D2} and Lemma \ref{Gauss} yields the desired inequality, that is,
\begin{align*}
\Big|\sum_{ n=1 }^{N}e(n^2t+nx)\Big| =
\frac{N}{\sqrt{q}}+o\Big(\frac{N}{\sqrt{q}}\Big).
\end{align*}
The proof of the theorem is complete.
\end{proof}
\begin{rem}If in the definition of the major arcs we set $\epsilon=0$, then it can be seen that the lower bounded is $N^{\frac{3\alpha}{4}-\frac{7}{4}}.$  previous proof can be modified as follows. We write 
\begin{align*}
&\sup_{x \in \T}\Big(\int_{0}^{1} \Big|\sum_{n=1}^{N}e(nx) e(n^2t)\Big|^\alpha dt\Big)\\ 
&\geq \sup_{x \in I(b,q)}\Big(\int_{\displaystyle \bigcup_{\overset{a=1}{a \wedge q=1}}^{q} I(a,q)} \Big|\sum_{n=1}^{N}e(nx) e(n^2t)\Big|^\alpha dt\Big)\\
&\geq \frac{N^{\alpha-2}}{\sqrt{q}}\phi(q)	
\end{align*}
By taking $q=\sqrt{N}$ and applying the Prime Number Theorem, we obtain 
\begin{align*}
&\sup_{x \in \T}\Big(\int_{0}^{1} \Big|\sum_{n=1}^{N}e(nx) e(n^2t)\Big|^\alpha dt\Big)\\ 
&\geq N^{\frac{3\alpha}{4}-\frac{7}{4}}.	
\end{align*}
Let us notice that our arguments yields that the following conjecture which is seems to be attributed to Bourgain does not holds. 
	\begin{conj}There exists a constant $\delta$ such that for any $N \in \N^*$, for any $p \in (2,4)$, we have
		$$\Big\|\sum_{k=1}^{N}a_k e(k^2\theta)\Big\|_p \ll \big(\log(N)\big)^{\delta}
		\Big\|\sum_{k=1}^{N}a_k e(k^2\theta)\Big\|_2 
		$$
	\end{conj}
\end{rem}
\section{Proof of Theorem \ref{main3}}
For the proof of Theorem \ref{main3}, we need the following Lemma from \cite[Exercises 6,7 and 8 of Chap. 3]{Apostol}. For sake of completeness, we gives its proof.
\begin{lem}\label{E}Let $N \geq 2$ and $\beta>0$. Then
\begin{align*}
	\sum_{n=1}^{N}\frac{\phi(n)}{n^\beta}= 
\begin{cases}
\frac{N^{2-\beta}}{(2-\beta)\zeta(2)}+\frac{\zeta(\beta-1)}{\zeta(\beta)}+
O(N^{1-\beta}\log(N)) & \textrm{if}\; \beta>1, \beta \neq 2.\\
\frac{N^{2-\beta}}{(2-\beta)\zeta(2)}+
O(N^{1-\beta}\log(N)) & \textrm{if}\; \beta \leq 1,\\
\frac{\log(N)}{\zeta(2)}+\frac{C}{\zeta(2)}-A+O(\frac{\log(N)}{N})  & \textrm{if}\; \beta=2,
\end{cases} 
\end{align*}
where $C$ is the Euler–Mascheroni constant and $A=\sum_{ n \geq 1} \frac{\mu(n)\log(n)}{n^2},$ $\mu$ is the M\"{o}bius function.  
\end{lem}
\begin{proof}Let us assume that $\beta>1$ and $\beta \neq 2$. Then,
\begin{align*}
\sum_{n=1}^{N}\frac{\phi(n)}{n^\beta}
=\sum_{n=1}^{N}\frac{1}{n^\beta}\sum_{d|n}\mu(d)\frac{n}{d},
\end{align*}
since $\phi(n)=\sum_{d|n}\mu(d)\frac{n}{d}.$ Changing the order of summation, we write
\begin{align*}
\sum_{n=1}^{N}\frac{\phi(n)}{n^\beta}&=\sum_{q,d,qd \leq N}\frac{\mu(d)}{d^\beta q^\beta}\\
&=\sum_{d \leq N}\frac{\mu(d)}{d^\beta} \sum_{q \leq \frac{N}{d} } \frac{1}{q^\beta}.
\end{align*}
Now, by the standard estimation of the Riemann series $\Big(\frac{1}{m^\beta}\Big)$, we have
\begin{align}
\sum_{n=1}^{N}\frac{\phi(n)}{n^\beta}
&=\sum_{d \leq N}\frac{\mu(d)}{d^\beta} \Big(\frac{N^{2-\beta}}{(2-\beta)d^{2-\beta}}+\zeta(\beta-1)+O\Big(\frac{N^{1-\beta}}
{d^{1-\beta}}\Big)\Big)\nonumber\\
&=\frac{N^{2-\beta}}{(2-\beta)}\sum_{d \leq N}\frac{\mu(d)}{d^2}+\zeta(\beta-1)\sum_{d \leq N}\frac{\mu(d)}{d^\beta}+O\Big(N^{1-\beta}\sum_{d \leq N}\frac{\mu(d)}{d}\Big) \nonumber\\
&=\frac{N^{2-\beta}}{(2-\beta)\zeta(2)}+\frac{N^{2-\beta}}{(2-\beta)}\sum_{d > N}\frac{\mu(d)}{d^2}\nonumber\\
&+\frac{\zeta(\beta-1)}{\zeta(\beta)}+\zeta(\beta-1)\sum_{d > N}\frac{\mu(d)}{d^\beta}+O\Big(N^{1-\beta}\log(N)\Big) \label{F3}.
\end{align}
The last equality is due to the fact that for any $\gamma>1$, we have
$$\sum_{d \leq N}\frac{\mu(d)}{d^\gamma}=\frac{1}{\zeta(\gamma)},$$
and 
$$\Big|\sum_{d \leq N}\frac{\mu(d)}{d}\Big| \leq \sum_{d \leq N}\frac{1}{d}=O(\log(N)).$$
From \eqref{F3} we deduce easily the desired equality. By the similar arguments, it a simple matter to establish the the two other formulas. The proof of the Lemma is complete.
\end{proof}
Now let us proceed to the proof of Theorem \ref{main3}. We first put $\epsilon=0$ in the definition of the major arcs (eq. \eqref{major}) and take $\alpha >2$. By applying Lemma \ref{Key} we write
\begin{align*}
	&\int_{0}^{1} \int_{0}^{1} \Big|\sum_{n=1}^{N}e(n^2t+nx)\Big|^\alpha dx dt \nonumber\\
	&\geq \sum_{q=1}^{\sqrt{N}}\sum_{b=1}^{q} \sum_{a=1,a \wedge q=1)}^{q}
	\int_{\mathcal{M}(q,a,b)}\big|\sum_{n=1}^{N}e(n^2t+nx)\big|^\alpha dx dt\\
	&\geq \sum_{q=1}^{\sqrt{N}}\sum_{b=1}^{q} \sum_{a=1,a \wedge q=1)}^{q} \frac{N^\alpha}{q^{\frac{\alpha}{2}}},
\end{align*}
since $\mathcal{M}(q,a,b)$ are disjoints. Whence,
\begin{align*}
\int_{0}^{1} \int_{0}^{1} \Big|\sum_{n=1}^{N}e(n^2t+nx)\Big|^\alpha dx dt 
\geq N^{\alpha-3} \sum_{q=1}^{\sqrt{N}} \frac{\phi(q)}{q^{\frac{\alpha}{2}-1}},
\end{align*}
At this point,  we are going to apply Lemma \ref{E}. Let $\beta=\frac{\alpha}{2}-1.$ Then, $\beta=2$ correspond to $\alpha=6.$ We thus get
\begin{align}\label{lsix-3}
\int_{0}^{1} \int_{0}^{1} \Big|\sum_{n=1}^{N}e(n^2t+nx)\Big|^\alpha dx dt 
\geq  C_\alpha N^{3} \log(N)
\end{align}
Now, assume $2 < \alpha \leq 4$. Then $\beta \in ]0,1]$. It follows from Lemma \ref{E} that we have 
\begin{align*}
\int_{0}^{1} \int_{0}^{1} \Big|\sum_{n=1}^{N}e(n^2t+nx)\Big|^\alpha dx dt 
\geq  C_\alpha N^{\frac{3 \alpha}{4}-\frac{3}{2}} 
\end{align*} 
To finish the proof, we consider the case $\alpha>4, \alpha \neq 6$, that is, $\beta>1, \beta \neq 2$. In this case, by Lemma \ref{E}, we conclude that 
\begin{align*}
\int_{0}^{1} \int_{0}^{1} \Big|\sum_{n=1}^{N}e(n^2t+nx)\Big|^\alpha dx dt 
\geq  C_\alpha N^{\frac{3 \alpha}{4}-\frac{3}{2}}.
\end{align*} 

\appendix
\section{An extension of Cordoba's theorem.}
In this appendix, we present an extension of Cordoba's theorem \cite[p.167]{Cordoba} based on Play-Littlewood inequalities combined with Bernstein-Zygmund inequalities. Our proof is self-contain and gives an alternative proof to Cordoba's proof.
\begin{prop}\label{CordoII}Let  $\{a_n\}$ be  monotonically decreasing sequence. Then, for any $\alpha \in [0,4)$, for any  positive integer $\ell$, there is a constant $C_{\alpha,\ell}$ such that, for any $N \geq 1$, we have
	\begin{align*}
	\Big\|\sum_{n=1}^{N}k^{2\ell}a_ke(k^2 \theta)\Big\|_\alpha \leq C_{\alpha,\ell} 
	\Big(\sum_{n=1}^{N}k^{4\ell}|a_k|^2\Big)^{\frac12}.
	\end{align*}
\end{prop}
For the proof of Proposition \ref{CordoII} , we start by strengthening \cite[Corollary, p.172]{Cordoba} as follows.
\begin{lem}\label{CordoSI} For any $\alpha \in [2,4)$ and any $\ell \in \N$, there is a constant $C_{\alpha,\ell}$ such that, for any $N \geq 1$, we have
	\begin{align*}
	\Big\|\sum_{n=1}^{N}k^{2\ell}e(k^2 \theta)\Big\|_\alpha \leq C_{\alpha,\ell}
	\Big(\sum_{n=1}^{N}k^{4\ell}\Big)^{\frac12}.
	\end{align*}
\end{lem}
We recall that by Faulhaber's formula, for any positive integer $\ell$, we have 
\begin{align}\label{Fau}
\sum_{ k=1}^{N}k^{\ell}=\frac{1}{\ell+1}\sum_{j=1}^{\ell}(-1)^j\binom{\ell+1}{j}B_jn^{\ell+1-j},
\end{align} 
where $B_j$ is the $j$-th Bernoulli number with the convention of $B_1=-\frac{1}{2}.$

We need also the following inequality due to S. Bernstein and A. Zygmund For the proof of Proposition \ref{CordoII}. For its proof, we refer to \cite[Theorem 3.13, Chapter X, p. 11]{Zygmund}.
\begin{lem}{[Bernstein-Zygmund inequality].}\label{Bernstein} For any $p \geq 1$, for any polynomial $P$ of degree $n$, we have
	
	$$\big\|P'\big\|_{p} \leq n \big\|P\big\|_{p},$$
	
	where $P'$ is the derivative of $P$. The equality holds if and only if $P(e^{ix})=M \cos(nx+\xi).$
\end{lem} 
An obvious generalization of Bernstein-Zygmund inequality is given by the following statement.
\begin{lem}\label{BZG} For any $p \geq 1$, for any polynomial $P$ of degree $n$, for any $k \geq 1$, we have
	
	$$\big\|P^{(k)}\big\|_{p} \leq \frac{n!}{(n-k)!} \big\|P\big\|_{p},$$
	
	where $P^{(k)}$ stand for the $k$-th derivative of $P$. 
\end{lem} 

\begin{proof}[\textbf{Proof of Lemma \ref{CordoSI}.}] Assume $\ell=1$ and apply Lemma \ref{Bernstein} to 
	get
	\begin{align*}
	\Big\|\sum_{n=1}^{N}k^2e(k^2 \theta)\Big\|_\alpha \leq N^2 \Big\|\sum_{n=1}^{N}e(k^2 \theta)\Big\|_\alpha 
	\end{align*}
	Therefore, 
	\begin{align*}
	\Big\|\sum_{n=1}^{N}k^2e(k^2 \theta)\Big\|_\alpha \leq N^2 C_\alpha \Big\|\sum_{n=1}^{N}e(k^2 \theta)\Big\|_2. 
	\end{align*}
	We further have
	\begin{align*}
	\Big(\sum_{n=1}^{N}k^4 \Big)\geq  c. N^5 , 
	\end{align*}
	for some constant $c>0$. Combining those inequalities, we conclude that 
	\begin{align*}
	\Big\|\sum_{n=1}^{N}k^2e(k^2 \theta)\Big\|_\alpha \leq C_\alpha \Big\|\sum_{n=1}^{N}k^2e(k^2 \theta)\Big\|_2. 
	\end{align*}
	The general case follows from Lemma \ref{BZG} combined with Faulhaber's formula \ref{Fau}. The details are left to the reader.
\end{proof} 
Proof of Proposition \ref{CordoII}, we need also the following classical Littelwood-Paley inequalities. For its proof, we refer for instance to \cite[Theorem 4.22, Chapter XV, p.233]{Zygmund}. 
\begin{lem}\label{LP}Let $\alpha \in (1,+\infty)$. Then, there exist $A_\alpha,B_\alpha>0$ such that, for any trigonometric polynomials $P$ on the circle, we have
	\begin{align*}
	A_\alpha \big\|P\big\|_\alpha\leq  \Bigg\|\Big(\sum_{j \in \Z}\big|S_j(P)\big|^2\Big)^{\frac12}\Bigg\|_\alpha
	\leq B_\alpha \big\|P\big\|_\alpha.
	\end{align*}
	where $S_{j}$ is the $j$-th dyadic partial sum of the Fourier series of $P$, defined by the formulas
	\[S_j(P)=\begin{cases}
	\ds \sum_{2^j \leq n < 2^{j+1}}\hat{P}(n)e(nx) & \textrm{if~~} j>0,\\
	\hat{P}(0) & \textrm{if~~} j=0,\\
	\ds \sum_{-2^{|j|} < n \leq - 2^{|j|-1}}\hat{P}(n)e(nx) & \textrm{if~~} j<0.
	\end{cases}\]
\end{lem}
We proceed now to the proof of Proposition \ref{CordoII}. 

\begin{proof}[\textbf{Proof of Proposition \ref{CordoII}.}]Let $\ell=1$ and put
	\begin{align*}
	P(\theta)&=\sum_{n=1}^{N}k^2a_ke(k^2 \theta)\\
	&=\sum_{n=1}^{N^2}\hat{P}(k)e(k \theta),	
	\end{align*}
	with 
	\[\hat{P}(\ell)=\begin{cases}
	k^2 a_k & \textrm{if~~} 1 \leq \ell=k^2 \leq N^2,\\
	0       & \textrm{if~not.}
	\end{cases}\]	
	By Lemma \ref{LP}, it is suffices to establish that 
	\begin{align*}
	\Bigg\|\Big(\sum_{j \geq 1}\big|S_j(P)\big|^2\Big)^{\frac12}\Bigg\|_\alpha \leq C_\alpha
	\|P\|_2.
	\end{align*}
	But, by the triangle inequality combined with Fubini theorem,  we have
	\begin{align*}
	\Bigg\|\Big(\sum_{j \geq 1}\big|S_j(P)\big|^2\Big)^{\frac12}\Bigg\|_\alpha
	&=\Bigg(\Bigg[\int_{0}^{1}\Big(\sum_{j \geq 1}\big|S_j(P)(\theta)\big|^2\Big)^{\frac{\alpha}2} d\theta\Bigg]^{\frac{2}{\alpha}}\Bigg)^{\frac12}\\
	&\leq
	\Bigg(\Bigg[\sum_{j \geq 1}\int_{0}^{1}\big|S_j(P)(\theta)\big|^\alpha d\theta\Bigg]^{\frac{2}{\alpha}}\Bigg)^{\frac12} 
	\end{align*}
	Therefore
	\begin{align*}
	\Bigg\|\Big(\sum_{j \geq 1}\big|S_j(P)\big|^2\Big)^{\frac12}\Bigg\|_\alpha
	&\leq 
	\Bigg(\sum_{j \geq 1}\big\|S_j(P)\big\|_\alpha^2\Bigg)^{\frac12}
	\end{align*}
	Now, by the definition of the projection $S_j$, we can write
	\begin{align*}
	S_j(P)(\theta)&=\sum_{ 2^{\frac{j}{2}} \leq k< 2^{\frac{j+1}{2}}}k^2a_k e(k^2\theta)\\
	&=\sum_{ 2^{\frac{j}{2}} \leq k< 2^{\frac{j+1}{2}}}a_k \Big(T_{k}(\theta)-T_{k-1}(\theta)\Big),
	\end{align*}
	where $T_l(\theta)=\sum_{m=1}^{l}m^2e(m^2\theta).$ We thus have
	\begin{align*}
	S_j(P)(\theta)&=\sum_{ 2^{\frac{j}{2}} \leq k< 2^{\frac{j+1}{2}}}a_kT_{k}(\theta)- \sum_{ 2^{\frac{j}{2}} \leq k< 2^{\frac{j+1}{2}}}a_{k}T_{k-1}(\theta),\\
	& =\sum_{ 2^{\frac{j}{2}} \leq k< 2^{\frac{j+1}{2}}}a_kT_{k}(\theta)- \sum_{ 2^{\frac{j}{2}}-1 \leq k< 2^{\frac{j+1}{2}}-1}a_{k+1}T_{k}(\theta).
	\end{align*}
	Hence
	\begin{align*}
	S_j(P)(\theta)=\sum_{ k\in D_j}\big(a_k-a_{k+1})T_{k}(\theta)+a_{\lfloor 2^{\frac{j+1}{2}} \rfloor}T_{\lfloor 2^{\frac{j+1}{2}}\rfloor}(\theta)-a_{\lceil 2^{\frac{j}{2}}\rceil}T_{\lceil 2^{\frac{j}{2}}\rceil}(\theta),
	\end{align*}
	where $D_j=[2^{\frac{j}{2}},2^{\frac{j+1}{2}})$ and, as customary,
	$\lfloor x \rfloor$ is  the greatest integer less than or equal to $x$ , $\lceil x \rceil$ is 
	the least integer greater than or equal to $x$.
	Consequently, by the triangle inequality, we get
	\begin{align}\label{LP-2}
	\big\|S_j(P)\big\|_\alpha
	&\leq \sum_{ k\in D_j}\big(a_k-a_{k+1})\big\|T_{k}\big\|_\alpha+\nonumber \\
	&a_{\lfloor 2^{\frac{j+1}{2}} \rfloor}\big\|T_{\lfloor 2^{\frac{j+1}{2}}\rfloor}\big\|_\alpha+a_{\lceil 2^{\frac{j}{2}}\rceil}\big\|T_{\lceil 2^{\frac{j}{2}}\rceil}\big\|_\alpha.
	\end{align}
	Now, by appealing to  Lemma \ref{LP}, we rewrite \ref{LP-2} as follows.
	
	\begin{align*}
	\big\|S_j(P)\big\|_\alpha
	&\leq  C_\alpha \Bigg(\sum_{ k\in D_j}\big(a_k-a_{k+1}) \sqrt{k}+\nonumber \\
	&a_{\lfloor 2^{\frac{n+1}{2}} \rfloor}\sqrt{\lfloor 2^{\frac{j+1}{2}}\rfloor}+a_{\lceil 2^{\frac{j}{2}}\rceil}\sqrt{\lceil 2^{\frac{j}{2}}\rceil}\Bigg).
	\end{align*}
	Taking into account that $(a_k)$ is a decreasing sequence, we obtain
	\begin{align*}
	\big\|S_j(P)\big\|_\alpha
	&\leq  C_\alpha 2^{\frac{j}{4}}\Bigg(\sum_{ k\in D_n}\big(a_k-a_{k+1})+
	a_{\lfloor 2^{\frac{j+1}{2}} \rfloor}+a_{\lceil 2^{\frac{j}{2}}\rceil}\Bigg)\\
	&\leq \big(2C_\alpha\big) 2^{\frac{j}{4}} a_{\lceil 2^{\frac{j}{2}}\rceil}\\
	&\leq \big(2C_\alpha\big) 2^{\frac{j}{4}} a_{\lfloor 2^{\frac{j}{2}}\rfloor}\\
	&\leq \big(2C_\alpha\big) \Bigg(\sum_{k \in D_{j-1}} |a_k|^2\Bigg)^{\frac12}
	\end{align*}
	Taking the sum over $j$, we get 
	\begin{align*}
	\sum_{j \geq 2} \big\|S_j(P)\big\|_\alpha
	C_\alpha' \sum_{j \geq 1} \Bigg(\sum_{k \in D_{j}} |a_k|^2\Bigg)^{\frac12}
	\end{align*}
	Now, we notice that by Parseval equality we have
	\begin{align*}
	\Bigg\|\Big(\sum_{j \geq 2}\big|S_{j-1}(P)\big|^2\Big)^{\frac12}\Bigg\|_2
	=\sum_{j \geq 1} \Bigg(\sum_{k \in D_{j}} |a_k|^2\Bigg)^{\frac12}.
	\end{align*}
	Applying again Lemma \ref{LP}, we conclude that there is a positive constant $K_\alpha$ such that 
	\begin{align*}
	\sum_{j \geq 2} \big\|S_j(P)\big\|_\alpha \leq 
	K_\alpha \Big\|P\Big\|_2.
	\end{align*}
	The general case follows by the same arguments  combined with Lemma \ref{BZG} and this achieve the proof of the proposition.
\end{proof}
\begin{thank}
	The author wishes to express his thanks to Igor Shparlinski and Mahesh Nerurka for stimulating discussions on the subject. 
\end{thank}

\end{document}